\documentclass[12pt, reqno]{amsart}
\usepackage[utf8]{inputenc}

\usepackage{amssymb}
\usepackage{graphics}
\usepackage{graphicx}
\usepackage{amscd}
\usepackage{color}
\usepackage{amsmath, amsthm, epsfig,  psfrag}
\usepackage{pinlabel}
\usepackage{hyperref}
\usepackage{subcaption}
\usepackage[margin=3cm, marginpar=2.5cm]{geometry}

\newtheorem{theorem}{Theorem}
\newtheorem{lemma}[theorem]{Lemma}

\newtheorem*{cond1}{Property F1}

\newtheorem*{cond2}{Property F2}

\theoremstyle{definition}

\newcommand{\CC}{\mathbb{C}}

\newcommand{\ti}{\tilde}

\title[On uniqueness of symplectic fillings of links of surface singularities]{On uniqueness of symplectic fillings \\ of links of some surface singularities}
\author{Olga Plamenevskaya}
\thanks{Partially supported by NSF grant DMS-1906260}
\address{Department of Mathematics, Stony Brook University, Stony Brook, NY,
11794}
\email{olga@math.stonybrook.edu}

\begin{document}

\begin{abstract} We consider the canonical contact structures on links of rational surface singularities with reduced fundamental cycle. These singularities can be characterized by their resolution graphs: the graph is a tree, and the weight of each vertex is no greater than its negative valency. The contact links are given by the boundaries of the corresponding plumbings. In \cite{PS}, it was shown that if the weight of each vertex in the graph is at most $-5$, the contact structure has a unique symplectic filling (up to symplectic deformation and blow-up); the proof was based on a symplectic analog of de Jong--van Straten's description of smoothings of these singularities. In this paper, we give a short self-contained proof of the uniqueness of fillings, via analysis of positive monodromy factorizations for planar 
 open books supporting these contact structures.
\end{abstract}

\maketitle

\section{Introduction}

In this note, we consider links of complex surface singularities, equipped with their canonical contact structures. Let $X \subset \CC^N$ be a singular complex surface with an isolated singularity at the origin. For small
$r > 0$, the intersection $Y = X \cap S_r^{2N-1}$ with the sphere 
$S_r^{2N-1}$ = $\{|z_1|^2 + |z_2 |^2 + \dots + |z_N|^2 = r\}$ is a
smooth $3$-manifold called the link of the singularity $(X, 0)$. The induced contact structure $\xi$ on $Y$ is the
distribution of complex tangencies to $Y$, and is referred to as the canonical or Milnor fillable contact
structure on the link. The contact manifold $(Y, \xi)$, which we will call the contact link, is independent
of the choice of $r$, up to contactomorphism.

Our main result, Theorem~\ref{main}, states that for a certain class of singularities, 
the canonical contact structure on the link has a unique symplectic filling (up to blow-up and symplectic deformation).
This theorem was originally proved in \cite{PS}; here, we will give a new proof, from a different perspective. 
The sufficient condition will be stated in terms of the dual resolution graph of the singularity. Recall that 
for  a normal surface singularity  $(X, 0)$, this graph is defined as follows. Consider a  resolution of the 
singularity, i.e. a proper birational morphism $\pi: \ti{X} \to X$ such that $\ti{X}$ is smooth.
 We can assume that 
the exceptional divisor $\pi^{-1} (0)$ has normal crossings. This means that $\pi^{-1} (0)= \cup_{v \in G} E_v$, where
the irreducible components $E_v$ are smooth complex curves that intersect transversally at double points only.  
The (dual) resolution graph encodes the topology of the resolution: the vertices  $E \in G$ correspond to the exceptional curves and are weighted by the self-intersection $E \cdot E$ of the corresponding curve, while the edges of $G$ record intersections of different irreducible components. Up to contactomorphism, the 
link of the singularity with its canonical contact structure can be reconstructed from the graph $G$ and the data of self-intersections and genera of exceptional curves,  as the boundary of the plumbing of symplectic disk bundles over surfaces according to $G$.

In this paper, we only work with {\em rational} singularities; then $G$ is always a tree, and each exceptional 
curve has genus 0. The following assumption plays the key role in this paper: for every exceptional curve $E$, 
we require that the self-intersection $E \cdot E$ and the valency $a(E)$ of the corresponding vertex in $G$ 
satisfy the inequality 
\begin{equation}\label{valency-weight}
a(E) \leq - E \cdot E. 
\end{equation}
Plumbing graphs with this property are sometimes referred to as ``graphs with no bad vertices''
in low-dimensional topology; a bad vertex, 
by definition, has valency greater than its negative weight. (The boundary of the corresponding plumbing is a Heegaard Floer L-space, \cite{OSz}.) If the dual resolution graph is a tree with the above property, $(X, 0)$ is a {\em rational singularity with reduced fundamental cycle}. In the literature, this type of singularities is also known as {\em minimal} singularities, \cite{Ko}.  

We will give a direct new proof of the following theorem, first established in~\cite{PS}: 

\begin{theorem}\label{main} \cite{PS} Suppose that $(X, 0)$ be a rational surface singularity with reduced fundamental cycle, and assume additionally that every exceptional curve in its resolution has self-intersection at most $-5$.  Then the contact link $(Y, \xi)$ of $(X, 0)$  has a unique minimal weak symplectic filling, which is Stein. 
\end{theorem}

In the special case where the resolution graph is star-shaped with three legs, this fact is proved in~\cite[Theorem 2.7, Remark 2.8]{BhSt}, by a different method.

Symplectic and Stein fillings of links of surface singularities are of interest because of the connection to algebro-geometric questions, namely to the smoothings of the singularity. The Milnor fiber of each smoothing of $(X, 0)$ gives a Stein filling of its link $(Y, \xi)$; another Stein filling can be provided by the minimal resolution of the singularity, after deforming the symplectic form. (Rational surface singularities are always smoothable, with an ``Artin smoothing component'' whose Milnor fiber gives the same Stein filling as the resolution. In particular, for the singularities in Theorem~\ref{main}, the filling can be viewed as the resolution or as the Milnor fiber for the Artin smoothing.)   An important question is whether {\em all} Stein fillings of a given surface singularity arise in this way \cite{Nem}. Although this correspondence breaks down when the singularity is sufficiently complicated \cite{Akh-Ozbagci1, Akh-Ozbagci2, PS}, the answer is positive for certain simple classes of singularities. Namely, all Stein fillings come   
from Milnor fibers or the minimal resolution for $(S^3,\xi_{std})$ \cite{Eliash}, for links of simple and simple elliptic  singularities \cite{OhtaOno1, OhtaOno2},   for lens spaces (links of  
cyclic quotient singularities) \cite{Li, NPP-cycl}, and in general for quotient singularities \cite{PPSU,Bhu-Ono}. Theorem~\ref{main} significantly extends this list. 

Our interest to the question of Theorem~\ref{main} was motivated by a (very special case of) a conjecture of Koll\'ar on deformations of rational surface singularities, \cite{Kol2}. The conjecture asserts  that every exceptional curve has self-intersection at most $-5$ in the resolution of a rational singularity, then  the base space of a
semi-universal deformation of this singularity has a unique component. For the case of rational singularities with reduced fundamental cycle, the conjecture was established by de Jong--van Straten \cite{dJvS}; in particular, it follows that under the hypotheses of Theorem~\ref{main}, the singularity has a unique smoothing component. Our Theorem~\ref{main} gives the symplectic analog of this statement. 

The proof we gave in~\cite{PS} comes as a side product of the theory developed in that article, where we describe symplectic fillings of the corresponding class of singularities via a symplectic analog of de Jong--van Straten's construction. Fillings are encoded by certain configurations of symplectic disks in $\CC^2$; we were then able to apply a lemma of de Jong--van Straten to establish ``combinatorial uniqueness'' of the corresponding disk arrangements, and then finish the argument via topological considerations.   

In this paper, we will instead give a direct proof of Theorem~\ref{main}, working with open book factorizations. As a corollary, we get a symplectic proof that all smoothings of the corresponding singularity are diffeomorphic. We will assume that the reader is familiar with the basics of open book decompositions for contact 3-manifolds; see \cite{Etn-ob} for a survey. Under the hypotheses of the theorem, the canonical contact structures on the links of singularities admit planar open books. (This follows from a construction of Gay--Mark \cite{GayMark}, see Section~\ref{s2}. Planarity was also a key ingredient that allowed us to build an analog of the de Jong--van Straten theory in \cite{PS}.) In the planar case, symplectic fillings can be studied via theorems of Wendl and Niederkruger--Wendl \cite{We, NiWe}: every minimal symplectic filling is symplectic deformation equivalent to a Lefschetz fibration over a disk with the same planar fiber $P$. The classification of fillings then reduces to enumerating positive factorizations of the monodromy of the given open book.

In general, finding all positive factorizations of the monodromy is a daunting task, even in the planar case. The question is much easier if one only seeks to determine the image of the Dehn twists of the factorization in the {\em abelianization} of the mapping class group of the page. This is equivalent to finding the {\em homology classes} of the curves about which the Dehn twists are performed; we also disregard the order of the twists. 
This easier question can be studied by counting how many times the Dehn twists enclose each hole in the planar page, and how many times they enclose each pair of holes. (The planar page is a disk with holes, and we say that a simple closed curve in a disk with holes encloses a hole if the the curve separates the hole from the outer boundary component of the disk.)  If $P$ is planar, any two factorizations of the boundary-fixing monodromy $\phi:P\to P$ can be connected by a sequence of lantern relations, and it follows that the number of Dehn twists enclosing a given hole (or a given pair of holes) is {\em independent} of the factorization of $\phi$. Thus, we can introduce the multiplicity $m(v)$ of a hole $v$ with respect to  the monodromy $\phi$, and similarly the joint multiplicity $m(v_1, v_2)$ of a pair of holes $v_1, v_2$. Knowing these multiplicities, one can attempt to describe possible other factorizations of
$\phi$, by examining the combinatorics of how the Dehn twists can enclose the holes. This method was introduced in \cite{PVHM} to classify fillings of certain lens spaces. 

Once we understand the Dehn twists in the factorization at the level of homology classes of the curves, additional information is needed to find the isotopy classes of the curves. In the case at hand, this step is possible because the given monodromy admits a positive factorization into Dehn twists about {\em disjoint} curves.  

In \cite{PS}, the combinatorial part of the proof was based on the description of fillings via a symplectic analog of the de Jong--van Straten construction. We then used the result of \cite{dJvS} asserting uniqueness of a combinatorial solution for a certain curve arrangement problem. For the second part of the proof, we gave a direct mapping class argument.  The purpose of the note is to give a direct multiplicity-count argument for the first part, see Lemma \ref{lem1}. For the second part, we essentially repeat the reasoning from \cite{PS}; this argument, based on right-veering properties, is given in Lemma \ref{lem2} for completeness.

It is interesting to note that our direct argument for the combinatorics of Dehn twists follows the strategy of 
\cite[Theorem 6.23]{dJvS}: we translate their proof from the incidence matrices to mutiplicities of holes, and provide some extra details where needed. 

{\bf Acknowledgements.} Originally, we proved Theorem~\ref{main} in joint work with Laura Starkston \cite{PS}. The alternative proof given here resulted from my attempts to understand the relation of the combinatorial constructions of \cite[Theorem 6.23]{dJvS} to open book factorizations. 

This article was written for Proceedings of the 2020 BIRS workshop on 
Interactions of Gauge Theory and Contact and Symplectic Topology. I have benefited greatly from the series of the BIRS workshops on this topics and would like to thank the organizers of all the past workshops of the series.

\section{Proof of Theorem~\ref{main}} \label{s2}

To begin, we recall the construction of the open books supporting the canonical contact structures for 
the class of singularities that satisfy~\eqref{valency-weight}, \cite[Theorem 1.1]{GayMark}.
Starting with the plumbing graph $G$, the construction given by Gay and Mark produces a planar Lefschetz fibration compatible with the symplectic resolution of a rational singularity $(X, 0)$ with reduced fundamental 
cycle. (The sympectic structure on the plumbing can be deformed to the corresponding Stein structure.) 
We describe the  induced planar open book on the link $(Y, \xi)$. To construct the page of the Gay--Mark open book, 
take a sphere $S_E$ for each vertex  $E \in G$ and cut out  $-a(E) - E \cdot E \geq 0$ disks out of this sphere. (As before, $a(E)$ is the valency
of the vertex $E$; the number of disks is non-negative by~\eqref{valency-weight}.) 
Next, make a connected sum of these spheres with holes  by adding a connected sum neck for each 
edge of $G$.  For a sphere $S_E$
corresponding to the vertex $E$, the number of necks equals the number of edges adjacent to $E$, i.e. its valency $a(E)$. 
The resulting surface $S$ has genus 0 because $G$ is a tree. See Figure~\ref{pic:GayMark} for an example. 
The open book monodromy is given by the product of positive Dehn twists around each of the holes and around the meridians of the necks. We will call this product the {\em standard} factorization of the Gay-Mark monodromy.

\begin{figure}[htb]
		\centering
		\bigskip
		\labellist
		\small\hair 2pt
		\pinlabel $-6$ at 187 38
		\pinlabel $-5$ at 207 38
		\pinlabel $-2$ at 228 38
		\pinlabel $-4$ at 158 38
		\pinlabel $-3$ at 177 57
		\pinlabel $-2$ at 163 68
		\pinlabel $-3$ at 200 68
		\pinlabel $-3$ at 172 18
		\endlabellist
		\includegraphics[scale=1.7]{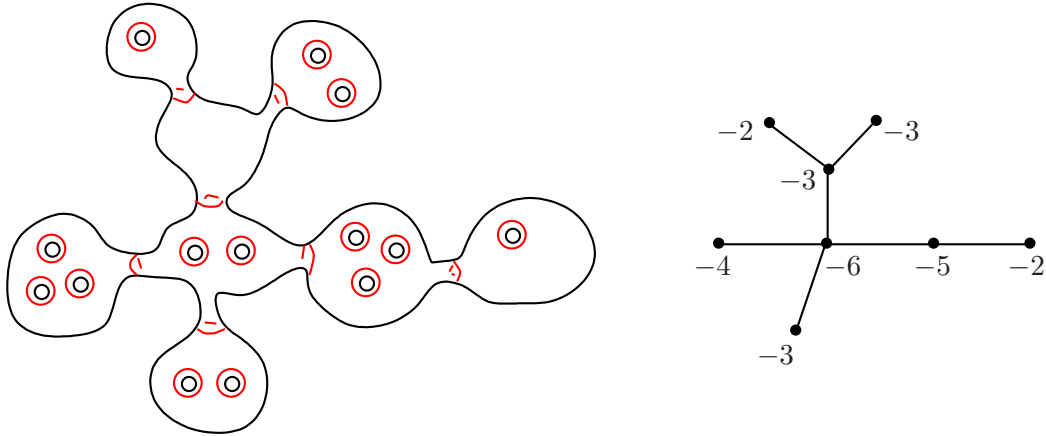}
		\caption{The Gay--Mark open book supporting the canonical contact structure on the link of the singularity with dual resolution graph shown on the right. The page of the open book has genus 0 and is constructed from the spheres with holes coresponding to the vertices of the graph. Each sphere is connected to the other spheres by necks that correspond to the edges; the total number of holes and necks for each sphere equals the negative self-intersection of the vertex. 
		The monodromy is the product of the positive Dehn twists about the boundaries of the holes and the meridians of the necks; these curves are shown in red.}
		\label{pic:GayMark}
	\end{figure}

To examine positive factorizations of this open book, we first we put the resolution graph in the following special form, as in \cite{dJvS}.  We choose the vertices $E_1, E_2, \dots E_k$ and partition the remaining vertices into subsets $R_2, \dots R_k$ as shown in Figure~\ref{arrange-graph}, so that for any vertex $F \in R_j$, the length $l(E_j, F)$ of the chain from $F$ to $E_j$ satisfies $l(E_j, F)\leq j-1$. Here, the length 
of chain means the number of edges; for example, the statement means that every vertex in $R_2$ is directly connected to $E_2$ by a single edge. This can always be achieved via the following procedure. We choose $E_1$ to be the endpoint of a longest chain $C$ in the graph; then $E_1$ is necessarily a leaf vertex of $G$. Let $E_2$ be its adjacent vertex, and let $E_3$ be the vertex on the chain $C$ that 
is adjacent to $E_2$. Removing $E_2$ from $G$, we get one connected component consisting of $E_1$,
another that contains $E_3$, and possibly a number of other vertices in the remaining components.
Let $R_2$ be the set of these remaining vertices. Each vertex $F \in R_2$ must be a leaf vertex connected to $E_2$ (otherwise we can build a chain longer than $C$ by going to $R_2$ instead of $E_1$); thus the condition $l(E_2, F)\leq 1$ is satisfied. If $E_3$ is a leaf vertex, it can be included in $R_2$, and the procedure is over. If $E_3$ is not a leaf vertex, and every other vertex in  $G \setminus (E_1\cup R_2)$ is connected to $E_3$ by a path of at most 2 edges, then all remaining vertices can be included in $E_3$. Otherwise we consider the vertex $E_4$ preceding $E_3$ on the path $C$. Removing $E_3$, we set aside the two components of $G\setminus E_3$ that contain $E_1$ resp. $E_4$ and let $E_3$ to be the set of all remaining vertices. Again, since no chain in the graph $G$ can be longer than $C$, every vertex  $F \in R_3$ must be connected to $E_3$ by a path of no more than 2 edges, satisfying $l(F, E_3)\leq 2$. We continue this process to define $E_5$, $R_4$, etc, stopping when we reach $k$ such that all the remaining vertices in $G$ can be connected to $E_k$ by a path no longer than $(k-1)$ edges, and thus placed in $R_k$. See Figure~\ref{arrange-graph}. We will say that vertices of the graph are {\em conveniently arranged} if they are partitioned into subsets as above.

\begin{figure}[htb]
		\centering
		\bigskip
		\labellist
		\small\hair 2pt
		\pinlabel $E_1$ at 3 4
		\pinlabel $E_2$ at 28 4
		\pinlabel $E_3$ at 72 4
		\pinlabel $E_4$ at 113 4
		\pinlabel $R_2$ at 18 38
		\pinlabel $R_3$ at 58 52
		\pinlabel $R_4$ at 107 58
		\endlabellist
		\includegraphics[scale=2]{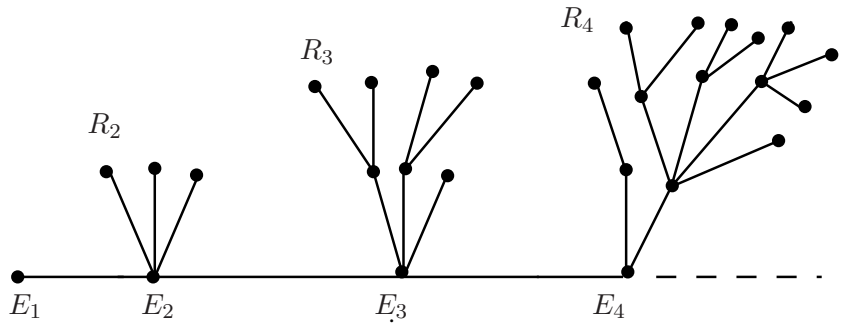}
		\caption{A graph with conveniently arranged vertices: after removing $E_1, E_2, \dots$, the remaining vertices are partitioned into subsets $R_2, R_3, \dots$. Every vertex in the set $R_j$ connects to the vertex $E_j$ by a chain with fewer than $j$ edges.}
		\label{arrange-graph}
	\end{figure}

For the resolution graph $G$ with conveniently arranged vertices, we build the Gay-Mark open book as in Figure~\ref{pic:GayMark}. We will identify 
the planar page of this open book with a disk with holes, so that outer boundary of the disk corresponds to the boundary of one of the holes associated to the vertex $E_1$. This identification and the choice of the outer boundary component of the disk will be fixed from now on, for the statement and the proof of Lemma~\ref{lem1}. In the standard factorization of the Gay-Mark monodromy, there is a sequence of the Dehn twists $D_1, D_2, \dots D_k$ around a nested collection of curves 
$\gamma_1, \dots \gamma_k$, such that 

(i) $D_1$ is the twist around the outer boundary component $\gamma_1$ of the page, and therefore $D_1$ encloses all the holes;

(ii) $D_j$ is the twist around the neck $\gamma_j$ between $E_{j-1}$ and $E_j$ for $j=2, \dots, k$, so that $D_j$ encloses all 
the holes corresponding to $E_j, R_j, E_{j+1}, R_{j+1}, \dots, E_k, R_k$.

The curves $\{\gamma_j\}_{j=1}^k$ cut the disk page into annular domains $\{V_j\}_{j=1}^{k-1}$ such that 
$V_j$ is bounded by $\gamma_j, \gamma_{j+1}$ for $j=1, \dots, k-1$, and  a disk $V_k$ bounded by $\gamma_k$. 
It follows that the joint multiplicity of any two holes from $V_j$ is at least $j$. 
In the Gay-Mark construction, the holes from  $V_j$ are associated to vertices $E_j, R_j$ of the graph $G$.

We now make a choice of a certain ordered subset of holes in the page.  
Because the valency of $E_1$ is one, and the self-intersection $E_1 \cdot E_1$ is at most $-5$, the corresponding annular domain $V_1$  contains  $-E_1\cdot E_1-2 \geq 3$  holes. We label three 
of these holes as $v_1^1$, $v_1^2$, $v_1^3$. Next, again because self-intersections of vertices are at most $-5$,
we can pick three holes  $v_2^1, v_2^2, v_2^3$ in the domain $V_2$. We require that $v_2^1, v_2^2, v_2^3$
satisfy an additional condition $m(v_2^r, v_2^s)= 2$: if $R_2$ is non-empty, we make sure  that no two holes are 
in the same branch of $R_2$ to avoid higher joint multiplicities. For $j=3, \dots k$, we proceed to pick $v_j^1, v_j^2, v_j^3$ in the domain $V_j$, choosing different branches of $R_j$ if $R_j$ is non-empty, so that $m(v_j^r, v_j^s)= j$ for any pair 
of indices $r, s=1, 2,3$.  By construction, we have 
\begin{equation}\label{mult}
m(v_i^r, v_j^s)= \min (i, j)
\end{equation}
for any two chosen holes $v_i^r, v_j^s$. 

The choice of the holes $v_1^1, v_1^2, v_1^3, \dots, v_k^1, v_k^2, v_k^3$ will be fixed. 
By construction, the standard factorization of the Gay--Mark open book satisfies

\begin{cond1} The factorization includes Dehn twists $D_1, \dots, D_k$ such that 
\begin{itemize}
 \item the Dehn twist $D_j$ encloses the holes $v_i^1, v_i^2, v_i^3$ for all $i\geq j$.
 \end{itemize}
\end{cond1} 
This is illustrated in Figure~\ref{nested-twists}. Note that we have only listed the Dehn twists that correspond to the edges of the chain $E_1, \dots, E_k$. The Dehn twists that correspond to edges the sets 
$R_j$ are not listed above; for each $R_j$, the corresponding Dehn twists enclose holes in the domain $V_j$. While these Dehn twists may be nested, the hypothesis that chains of edges in $R_j$ connecting to $E_j$ have length at most $j-1$ gives a bound on a number of twists enclosing any hole $w$ in $V_j$ in the standard factorization: there are at most $j$ nested Dehn twists inside $V_j$ (including the Dehn twist around the boudary of $w$), in addition to $D_1, \dots, D_j$. This means that in $\phi$, the multiplicity 
of the hole $w$ is at most $2j$.

For an inductive step in our proof, we will consider graphs $G$ satisfying a weaker hypothesis: when the vertices of $G$ are conveniently arranged, we require that the self-intersection  $E_k\cdot E_k \leq -4$, whereas $E \cdot E\leq -5$ for every other vertex $E \in G$. Note that in this case, we can still choose a labeled collection of holes 
$v_1^1, v_1^2, v_1^3, \dots, v_k^1, v_k^2, v_k^3$ as above.  Because $E_k \cdot E_k=-4$, 
we can use the lantern relation to replace the product of $D_k$ and three other Dehn twists (around holes or necks in the sphere corresponding to $E_k$) by three Dehn twists $D_{k}^1, D_{k}^2, D_{k}^3$. For this new factorization, we have:

\begin{cond2}  The factorization includes Dehn twists $D_1, \dots, D_{k-1}, D_{k}^1, D_{k}^2, D_{k}^3$ such that 
\begin{itemize}
\item   $D_j$ encloses the holes $v_i^1, v_i^2, v_i^3$ for all $i\geq j$, for each for each $j=1, \dots, k-1$;
\item  $D_{k}^1$ encloses  $v_k^2, v_k^3$ but not $v_k^1$;
\item $D_{k}^2$ encloses  $v_k^1, v_k^3$ but not $v_k^2$;
\item $D_{k}^3$ encloses  $v_k^1, v_k^2$ but not $v_k^3$.
\end{itemize}
\end{cond2}

\begin{figure}[htb]
		\centering
		\bigskip
		\labellist
		\small\hair 2pt
		\pinlabel $v_1^1$ at 28 88
		\pinlabel $v_1^2$ at 47 93
		\pinlabel $v_1^3$ at 68 92
		\pinlabel $v_2^1$ at 33 75
		\pinlabel $v_2^2$ at 48 79
		\pinlabel $v_2^3$ at 73 79
		\pinlabel $v_{k-1}^1$ at 30 42
		\pinlabel $v_{k-1}^2$ at 49 49
		\pinlabel $v_{k-1}^3$ at 72 47
		\pinlabel $v_k^2$ at 61 39
		\pinlabel $v_k^1$ at 50 27
		\pinlabel $v_k^3$ at 61 28
		\pinlabel $D_1$ at 11 82
		\pinlabel $D_2$ at 19 76
		\pinlabel $D_{k-1}$ at 25 59
		\pinlabel $D_k$ at 34 33
		
		\pinlabel $v_1^1$ at 179 88
		\pinlabel $v_1^2$ at 194 91
		\pinlabel $v_1^3$ at 223 93
		\pinlabel $v_2^1$ at 176 76
		\pinlabel $v_2^2$ at 202 80
		\pinlabel $v_2^3$ at 225 77
		\pinlabel $v_{k-1}^1$ at 180 54
		\pinlabel $v_{k-1}^2$ at 203 57
		\pinlabel $v_{k-1}^3$ at 218 56
		\pinlabel $v_k^1$ at 196 28
		\pinlabel $v_k^2$ at 193 39
		\pinlabel $v_k^3$ at 209 38
		\pinlabel $D_1$ at 155 83
		\pinlabel $D_2$ at 164 79
		\pinlabel $D_{k-1}$ at 170 61
		\pinlabel $D_k^3$ at 181 40
		\pinlabel $D_k^1$ at 220 39
		\pinlabel $D_k^2$ at 199 20
		
		\endlabellist
		\includegraphics[scale=1.9]{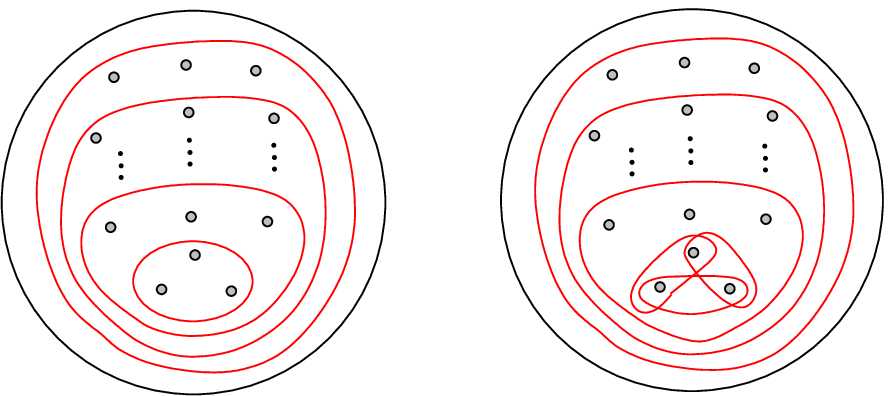}
		\caption{Property F1 (left) and Property F2 (right) for a chosen subset of holes $v_1^1, v_1^2, v_1^3, \dots v_k^1, v_k^2, v_k^3$ and the Dehn twists that enclose them in a factorization. Note that we only require that the {\em homology classes}  of the curves are as schematically shown; isotopy classes may look different from the picture.} 
		\label{nested-twists}
	\end{figure}

Under the hypotheses of the following lemma,  we will show that an {\em arbitrary} factorization of the monodromy of the Gay--Mark open book must have Property F1 or Property F2. This will be a step in the argument 
showing that any monodromy factorization must be standard if the self-intersection of each vertex of $G$ is at most $-5$.

\begin{lemma} \label{lem1} Suppose that the vertices of the graph $G$ are conveniently arranged, with distinguished vertices $E_1, E_2, \dots E_k$, and the corresponding sets $R_2, \dots, R_k$.  Assume that the self-intersections of all vertices in the graph  are at most $-5$, except possibly $E_k$, which has self-intersection at most $-4$. Suppose also that there is a collection of holes $\{v_j^1, v_j^2, v_j^3\}_{j=1}^k$, chosen as above.  Then 
 \begin{enumerate}

\item If $E_k \cdot E_k = -4$, then {\em every} monodromy factorization includes

\begin{enumerate}
\item the Dehn twist around the outer boundary component of the page;

\item a collection of Dehn twists (containing the twist around the outer boundary) that has Property F1 or Property F2.
\end{enumerate}

\item If $E_k \cdot E_k \leq -5$, then every monodromy factorization is homologically equivalent to the standard one.
\end{enumerate} 
 
\end{lemma}

\begin{proof} We will build an inductive argument, with double induction on $k$  and the number of vertices in the graph. 

The base of induction is given by $k=1$ and $k=2$. The case of $k=1$ corresponds to lens spaces and was treated in \cite{PVHM}. (This is an easy exercise on computing multiplicities.) The case $k=2$ is straightforward but more tedious; we check it after explaining the induction step.

For now, assume that $k \geq 2$, and that the statement of the lemma is established for all graphs where the chain of distinguished vertices $E_1, E_2, \dots $ has length at most $k$. Consider the graph $G$ with convenienly arranged vertices with a longer chain  $E_1, \dots, E_{k+1}$, and the remaining vertices partitioned into the  sets $R_2, \dots, R_{k+1}$.  Take a new graph $G'$,  obtained from $G$ by removing all vertices of $R_{k+1}$ and $E_{k+1}$, increasing by $1$ the self-intersection of $E_k$, and keeping the same self-intersection for all other vertices. The Gay-Mark open books $(P, \phi)$ and $(P', \phi')$, representing respectively the contact links of singularities with graphs $G$ and $G'$, are related as follows. To obtain the page  $P'$ from $P$, we cap off 
all the holes in $P$ associated to $E_{k+1}$ and $R_{k+1}$; the boundary-fixing diffeomorphism $\phi$ then induces $\phi'$. (Note that $E_k \cdot E_k$ increases by 1 in the graph $G'$ since the corresponding subsurface in the page has fewer boundary components now: removing $E_{k+1}, R_{k+1}$ is the same as pinching off the neck connecting the $E_k$-sphere to the $E_{k+1}$-sphere.) 

Fix an arbitrary factorization $\Phi$ of the Gay-Mark open book for $G$. When the holes in $P$ are capped off to obtain $P'$, $\Phi$ induces the factorization $\Phi'$ for the open book $(P', \phi')$. Since by assumption 
$E_k \cdot E_k \leq -5$ in $G$, the self-intersection of the corresponding vertex is at most $-4$ in $G'$.
The induction hypothesis applies to the graph $G'$, and therefore, the conclusion of the lemma holds for the factorization $\Phi'$ of the monodromy $\phi'$. In particular, there is a Dehn twist
$T'=T_1'$ around the outer boundary component of $P'$ in the factorization $\Phi'$, and moreover, there are Dehn twists $T_2', T_3', \dots, T_{k-1}'$, and $T'_k$ (or $T'_{k,1}$, $T'_{k,2}$, $T'_{k,3}$) that have Property F1 (or, respectively, Property F2). These Dehn twists must be induced by the corresponding Dehn twists
$T=T_1, T_2, \dots, T_{k-1}$ and $T_k$ (or $T_{k,1}$, $T_{k,2}$, $T_{k,3}$) in the factorization $\Phi$ of 
$\phi:P\to P$.

To show that $\Phi$ has a Dehn twist around the outer boundary component of $P$, we need to check that $T$ encloses all the holes corresponding to $E_{k+1}$ and $R_{k+1}$
that were removed from $P$ to obtain $P'$; we already know that $T$ encloses all the holes that $P$ inherits from $P'$.
 For the sake of contradiction, let $v=v_{k+1}^s$ be a hole associated to  $E_{k+1}$ or $R_{k+1}$,
and suppose that it is not enclosed by $T$.  First assume that  the factorization $\Phi'$ has property~F1, so that the factorization $\Phi$ includes Dehn twists $T=T_1, T_2, \dots, T_{k-1}, T_k$ as above.
We examine the  multiplicities of the selected holes. 
Because these multiplicities can be computed from  the standard factorization of $\phi:P\to P$, by~\eqref{mult} we know that the joint multiplicity $m(v, v^s_j)=j$ for $i=1, 2, 3$ and $j=, \dots, k$.

If $T=T_1$ does not enclose $v$, the holes $v$ and $v_j^s$ are enclosed together by at most $j-1$ of the Dehn twists  $T_1, T_2, \dots, T_{k-1}, T_k$. Even if $v$ is enclosed by all of $T_2, \dots, T_{k-1}, T_k$,
it follows that there must be an additional Dehn twist $\tau_j^s$ enclosing 
both $v$ and $v_j^s$. Observe that the Dehn twists $\tau_j^i$ must be all distinct (that is, $\tau_j^s=\tau_i^r$ only if $i=j, r=s$): the joint multiplicity $m(v_j^s, v_i^r)$ of any two distinct holes $v_j^s, v_i^r$ is already realized by $T_1, T_2, \dots, T_{k-1}, T_k$, so no additional Dehn twist can enclose them both. It follows that the hole $v$ must be enclosed by at least $3k$ distinct Dehn twists $\tau_1^1, \tau_1^2, \tau_1^3, \dots, \tau_{k}^1, \tau_k^2, \tau_k^3$ in the factorization $\Phi$, 
in addition to $k-1$ Dehn twists $T_2, \dots, T_{k-1}, T_k$. It is not hard to see that if $v$ is not enclosed by some of the twists among $T_2, \dots, T_{k-1}, T_k$, then each missing twist will need to be replaced by several 
individual twists to achieve $m(v, v^s_j)=j$. It follows that $v$ is enclosed by at least $3k+k-1=4k-1$ twists.
To obtain a contradiction, we compute the multiplicity $m(v)$ in the monodromy $\phi$. The hole $v$ is associated to $E_{k+1}$ or to some vertex $E$ in $R_{k+1}$; in the standard factorization of $\phi$, it is enclosed by the small twist around the hole $v$, by the outer boundary twist, as well as by the Dehn twists corresponding to the edges in the chain in $G$ from $E_1$ to $E_{k+1}$ and then the chain 
from $E_{k+1}$ to $E$, if the latter chain is present. 
Since $E \in R_{k+1}$, and by construction the length of the chain from $E_{k+1}$ to 
any vertex in $R_{k+1}$ is at most $k$, we see that $m(v)\leq 2k+2$. This is a contradiction since
 $2k+2<4k-1$ for $k\geq 2$.

Similar reasoning leads to the same conclusion in the case where the factorization $\Phi'$ has Property~F2   instead of Property~F1. As above, we see that if $v$ is not enclosed by $T=T_1$, there must be at least $3(k-1)$  distinct Dehn twists $\tau_1^1, \tau_1^2, \tau_1^3, \dots, \tau_{k-1}^1, \tau_{k-1}^2, \tau_{k-1}^3$ in the factorization $\Phi$ to achieve $m(v, v^s_j)=j$ for $j=1, \dots, {k-1}$, $s=1,2,3$.
The holes  $v_k^1, v_k^2, v_k^3$ need a bit more attention. Indeed, when $v$ is not enclosed by $T=T_1$, there are at most $k-2$ twists 
among $T_2, \dots, T_{k-1}$ enclosing $v$ and $v_k^s$ for each $s=1, 2, 3$. The joint multiplicity 
$m(v, v_k^s)=k$ can be achieved if $v$ is enclosed by all three twists $T_{k,1}$, $T_{k,2}$, $T_{k,3}$, in addition to all of  $T_2, \dots, T_{k-1}$ and $\tau_1^1, \tau_1^2, \tau_1^3, \dots, \tau_{k-1}^1, \tau_{k-1}^2, \tau_{k-1}^3$. This would give $m(v)\geq 3k+(k-1)$ as before. Another case is when one of the twists $T_{k,1}$, $T_{k,2}$, $T_{k,3}$ (say $T_{k,3}$) does not enclose $v$. In that case, two additional twists, distinct from all of the above, enclosing respectively $v$ and $v_k^1$ (but not $v_k^2$ or $v_k^3$) and $v$ and $v_k^2$(but not $v_k^1$ or $v_k^3$), are needed  (again for the reason of joint multiplicities). 
This would still yield $m(v)> 3k+(k-1)$. As in the case of Property~F1, the multiplicity of $v$ will be even higher if $v$ is not enclosed by some of the twists among $T_2, \dots, T_{k-1}, T_k$. As above, we get a contradiction since 
$m(v)\leq 2k+2$, as computed from the standard factorization of $\Phi$.

At this point, we have shown that the Gay--Mark open book for the resolution graph $G$ must have an outer boundary twist $T$ in any factorization $\Phi$, assuming that the smaller graph $G'$ satisfies the conclusion of the lemma. To prove the other statements of the lemma for $G$, we will now reduce to a different smaller graph $\ti{G}$.

In the page $P$, consider all the holes associated to the vertex $E_1 \in G$.  We know that these holes have joint multiplicity $1$ with any other hole in $P$, thus they cannot be enclosed by any twists other than $T$ that involve several holes. Since there is one boundary Dehn twist $\delta_i$ around each of  these holes in the standard factorization, so that the multiplicity of each hole is $2$, these boundary twists must be present in $Phi$ as well.  It follows that the factorization $\Phi$ has the form $\Phi=T \delta_1 \delta_2 \dots \delta_m \circ \ti{\Phi}$, 
where $\ti{\Phi}$ is supported in $\ti{P}= P \setminus V_1$, the part of the page $P$ associated to $G\setminus E_1$. Remove the vertex $E_1$ and its connecting edge from the graph $G$, and consider the  resulting graph $\ti{G}$, keeping the same self-intersections of vertices. Clearly, $\ti{\Phi}$ gives a factorization of the Gay--Mark open book associated to $\ti{G}$. 
If self-intersections of all vertices of $G$ are at most $-5$, the same holds for $\ti{G}$. The graph 
$\ti{G}$ has fewer vertices, so by the induction hypothesis, the factorization $\ti{\Phi}$ must be standard. It follows that the factorization $\Phi$ of the Gay-Mark open book for $G$ is standard as well, proving part (2) of the lemma. 

To make the induction step work for part (1), assume that in $G$, the vertex $E_{k+1}$ has self-intersection at most $-4$, while  all the other vertices have self-intersection at most $-5$. When we remove $E_1$ to form the graph $\ti{G}$, the vertices of $\ti{G}$ may no longer be conveniently arranged; after vertices are rearranged, 
we need to have the $(-4)$ vertex at the appropriate position to apply the induction hypothesis.
We must rearrange the vertices of $\ti{G}$ to have a chain $\ti{E}_1, \ti{E}_2, \dots$, with the other vertices partitioned into the sets $\ti{R}_1, \ti{R}_2, \dots$, so that the length of any chain in $E_j, R_j$ is at most $j-1$. Consider the vertex $E_2$ in $G$. If 
$R_2$ was not empty in $G$, we can pick a vertex of $R_2$ to play the role $\ti{E}_1$ in $\ti{G}$, let $\ti{R}_2$
to be the remaining vertices of $R_2$, and set $\ti{E}_2=E_2$, $\ti{E}_3=E_3, \dots$, $\ti{R_3}=R_3$,  $\ti{R_4}=R_4, \dots$. In this case, the $(-4)$ vertex $E_{k+1}$ in $G$ becomes the vertex $\ti{E}_{k+1}$ at the end of the chain $\ti{E}_1, \ti{E}_2, \dots$ in $\ti{G}$, as required. If $R_2$ is empty in $G$, we check if $R_3$ has a chain of (maximum possible) length $2$. If so, we rearrange the vertices: let $\ti{E}_3=E_3$, 
pick $\ti{E}_1$ and $\ti{E_2}$ forming a length 2 chain in $R_3$ (with $\ti{E}_1$ being the leaf vertex). 
Let $\ti{R_2}$ consist of all vertices other than  $\ti{E}_1$, $\ti{E}_3$, and let $\ti{R_3}$ consist of all remaining vertices of $R_3$, together with the old vertex $E_2$. For $j \geq 4$, we have $\ti{E}_j=E_j$, 
$\ti{R_j}=R_j$, so the $(-4)$ vertex remains in the right place for the graph $\ti{G}$, which is now conveniently arranged. 
See Figure~\ref{flip-graph}.
If there are no chains of length 2 in $R_3$,
we similarly examine $R_4$ to see if there are chains of length 3. If so, we flip the graph to put this length 3 chain into the position of vertices  $\ti{E}_1, \ti{E_2}, \ti{E_3}$, make the vertices $E_2, E_3$ and all of $R_3$ to be part of the new set $R_4$; the graph is now conveniently arranged, and the $(-4)$ vertex does not move.  If there are no length 3 chains in $R_4$, we look at $R_5$, etc. To summarize, the above procedure means that we can conveniently rearrange the vertices of $\ti{G}$ without moving the $(-4)$ vertex whenever for some $j=2, \dots, k$, the set $R_j$ in $G$ has a chain of the maximum possible length $j-1$. If such a chain does not exist, we check if $R_{k+1}$ has a chain of length $k$. If so, this chain will become the new chain $\ti{E}_1, \ti{E_2}, \dots$, the old vertices $E_2, \dots, E_k$ as well as the sets $R_2, \dots, R_k$ will all be in 
$\ti{R}_{k+1}$, and the graph will be conveniently rearranged without moving the $(-4)$ vertex $E_{k+1}=\ti{E}_{k+1}$. Lastly, if each chain $R_j$, $j=2, \dots, k+1$ has length at most $j-2$ in $G$, we can set 
$\ti{E}_1=E_2$, $\ti{E}_2=E_3, \dots$, $\ti{E}_{k}=E_{k+1}$ and $\ti{R}_2=R_3, \dots$, $\ti{R}_{k}=R_{k+1}$, so that the graph $\ti{G}$ will be conveniently arranged, and the vertex $\ti{E}_{k} \in \ti{G}$ will have self-intersection $-4$.  
\begin{figure}[htb]
		\centering
		\bigskip
		\labellist
		\small\hair 2pt
		\pinlabel $G$ at 2 157
		\pinlabel $\ti{G}$ at 2 57
		\pinlabel $E_1$ at 2 105
		\pinlabel $E_2$ at 30 105
		\pinlabel $E_3$ at 58 105
		\pinlabel $E_4$ at 100 105
		\pinlabel $E_{k+1}$ at 172 105
		\pinlabel $R_3$ at 30 155
		\pinlabel $R_4$ at 90 160
		\pinlabel $R_{k+1}$ at 161 165
		\pinlabel $\ti{E}_1$ at 6 3
		\pinlabel $\ti{E}_2$ at 30 3
		\pinlabel $\ti{E}_3$ at 58 3
		\pinlabel $\ti{E}_4$ at 100 3
		\pinlabel $\ti{E}_{k+1}$ at 172 3
		\pinlabel $\ti{R}_2$ at 21 36
		\pinlabel $\ti{R}_3$ at 52 44
		\pinlabel $\ti{R}_4$ at 90 57
		\pinlabel $\ti{R}_{k+1}$ at 161 60
		\endlabellist
		\includegraphics[scale=1.7]{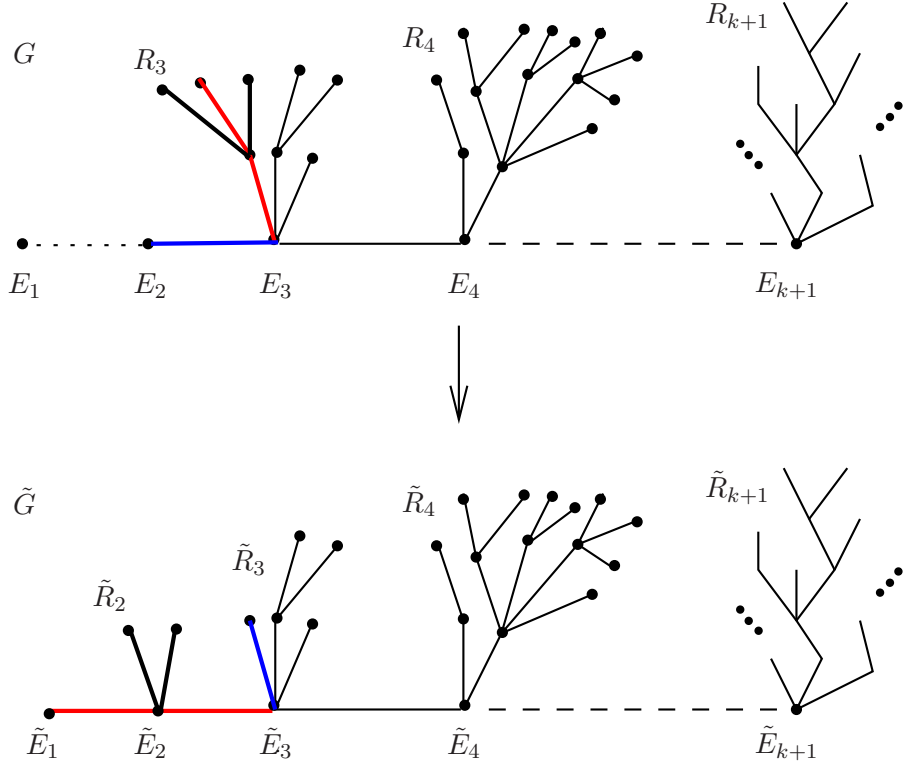}
		\caption{After removing the vertex $E_1$ from $G$, we flip a chain in the graph to make the new graph $\ti{G}$ conveniently arranged, while keeping in place the vertex $E_k$. The graph $G$ is shown at the top, the new graph $\ti{G}$ at the bottom. The picture illustrates the situation where $R_2$ is empty, and we flip 
		a length 2 chain in $R_3$, together with all the edges and vertices attached to this chain in $R_3$. The vertex $E_2$ becomes part of the new set~$\ti{R}_3$.}
		\label{flip-graph}
	\end{figure}

	With the rearrangement in place, part (1) follows by induction: if the factorization $\Phi$ of the Gay-Mark open book for $G$ has the form $\Phi=T \delta_1 \delta_2 \dots \delta_m \circ \ti{\Phi}$, where $\ti{\Phi}$ is the factorization 
of the Gay-Mark open book for $\ti{G}$, and part (1) of the lemma holds for $\ti{\Phi}$, then clearly the same is true for $\Phi$.

We now return to the base of induction and check the case  $k=2$. In this case, the graph is star-shaped with legs of length 1, with $E_2$ in the center. As above, the page $P$ is identified with the disk whose outer boundary corresponds to one of the holes associated to $E_1$;  
there are at least three holes $v_1^1, v_1^2, v_1^3$ in $V_1$. First, we claim that any factorization has a Dehn twist enclosing all of these holes. If not, we must have two distinct Dehn twists $\tau_1$ and $\tau_2$,
$\tau_1$ enclosing $v_1^1, v_3^1$ and $\tau_2$ enclosing $v_2^1, v_3^1$, because $m(v_1^r, v_1^s)=1$.  Since $m(v_3^1)=2$ and there are at least 4 holes in $P$ having joint multiplicity 2 with $v_3^1$, one of these Dehn twists, say $\tau_1$, contains an additional hole $w$. But then there must be two additional Dehn twists in the factorization, 
enclosing  respectively $v_2^1, v_1^1$ and $v_2^1, w$, which is impossible since $m(v_2^2)=2$. Thus, there is a Dehn twist $\tau$ enclosing $v_1^1, v_1^2, v_1^3$. We would like to show that $\tau$ encloses all the holes in $P$. Suppose not, and let $w$ be a hole not enclosed by $\tau$. Then there must be distinct Dehn twists $\tau_, \tau_2, \tau_3$, 
enclosing respectively $w,v_1^1$, $w, v_1^2$, $w, v_1^3$. Since $m(v_1^1)=m(v_1^2)=m(v_1^3)=2$, there cannot be any other Dehn twists enclosing  $v_1^1, v_1^2, v_1^3$, and since $m(v_1^r, v)=1$ for $r=1,2,3$ and any other hole $v$, every hole $v$ must be either in $\tau$ or in {\em each} of $\tau_1, \tau_2, \tau_3$ (but not simultaneously in $\tau$ and $\tau_i$, $i=1,2,3$). For the hole $w$, two cases are possible: (1) $w$ belongs to $E_2$, in which case $m(w)=3$, 
so there are no Dehn twists except $\tau_1, \tau_2, \tau_3$ enclosing $w$; or (2) $w$ belongs to $R_2$, in which case 
$m(w)=4$, and there is exactly one additional Dehn twist $\tau'$. In either case, there must exist another hole $w'$ such that $m(w, w')=2$; however, we can only get $m(w, w')=3$ (if $w'$ is in all three of $\tau_1, \tau_2, \tau_3$) or 
$m(w, w')=1$ (if $w'$ is in $\tau$ and $\tau'$). It follows that the Dehn twist $\tau$ must enclose all holes in $P$. We conclude that the factorization includes Dehn twists around the curves that are homologous, and therefore isotopic, to the boundaries of all the holes associated to $E_1$. As above, these can be removed from consideration. The same argument works for any leaf vertex of the graph, reducing the question to the situation of only one vertex, $E_2$. 
This is the case $k=1$ representing an open book for a lens space as in \cite{PVHM}; if $E_2\cdot E_2\leq -5$, there is a unique factorization, and if $E_2 \cdot E_2 =-4$, then the only other option for the homology classes of curves comes  from the lantern relation.   
\end{proof}

By Lemma~\ref{lem1}, we now know that under the hypotheses of Theorem~\ref{main}, the Dehn twists in every positive factorization are performed about the curves in the same {\em homology} classes as the Dehn twists in the standard factorization. We now show that the curves are in the same {\em isotopy} classes.

\begin{lemma} \label{lem2} Let $(P, \phi)$ be a planar open book whose monodromy $\phi$ admits a factorization $\Phi$ into a product of positive Dehn twists about {\em disjoint} simple closed curves in $P$. Suppose that $\Phi'$ is another positive 
factorization of $\phi$, such  that $\Phi$ is homologically equivalent to $\Phi'$.  Then the factorizations $\Phi$ and $\Phi'$ are the same, up to the order of Dehn twists. 
\end{lemma}

\begin{proof} 
After reordering, we can write $\Phi=D_1 D_2 \dots D_l \delta_1 \delta_2 \dots \delta_n$, where $\delta_i$'s  are Dehn twists about the boundary-parallel curves, and   $D_1, D_2, \dots, D_l$ are the Dehn twists around disjoint curves 
 $\gamma_1, \dots, \gamma_l$ in $P$ 
that are not boundary parallel. 

Then, again after reordering, we have $\Phi'=T_1 T_2 \dots T_l \delta_1 \delta_2 \dots \delta_n$, where the Dehn twists $D_j$ and $T_j$ are performed about homologous curves in $P$:
indeed, every boundary-parallel curve $\gamma_j$ is determined by its homology class, uniquely up to isotopy.  
We can thus remove the Dehn twists $\delta_1, \delta_2, \dots, \delta_n$ from consideration. We will use the same notation, 
$\Phi=D_1 D_2 \dots D_l$ and  $\Phi'=T_1 T_2 \dots T_l$ for the two factorizations of the diffeomorphism 
\begin{equation}\label{eq:non-bdry}
 \phi=D_1 D_2 \dots D_l=T_1 T_2 \dots T_l.
\end{equation}

We will prove the lemma by induction on the number $l$ of the non-boundary parallel Dehn twists. Identifying $P$ with a disk with holes, 
we can assume  that $\gamma_l$ is an innermost curve in the collection $\gamma_1, \gamma_2, \dots, \gamma_l$. Suppose that $\gamma_l$ encloses $r$ holes. 
Choose a  collection of arcs $\eta_1, \eta_2, \dots \eta_{r-1}$ connecting these holes and disjoint from  $\gamma_l$, 
so that after cutting along these arcs, the holes become a single hole, and the domain enclosed by $\gamma_l$  becomes an annulus (which deformation retracts to $\gamma_l$).  See Figure~\ref{cutting-disk}.  By construction, the arcs $\eta_1, \eta_2, \dots \eta_{r-1}$ are disjoint from the support of each of the Dehn twists $D_1, D_2, \dots, D_l$, thus the diffeomorphism $\phi= D_1 D_2 \dots D_l$ fixes each of these arcs.

 As in \cite[Proposition~3]{BMVHM} and  \cite[Section 2]{Foss}, we now make the following key observation:  after an isotopy removing non-essential intersections, 
all arcs $\eta_1, \dots, \eta_{r-1}$ must be also disjoint from the support of each of the Dehn twists $T_1, T_2, \dots, T_l$.  To see this, we recall that each 
right-handed Dehn twist is a right-veering diffeomorphism of the oriented surface $P$, \cite{HKM}. If $\alpha$ and $\beta$ are two arcs with the same endpoint $x \in \partial S$, we say that $\beta$
lies to the right of $\alpha$ if the pair of tangent vectors $(\dot{\beta},\dot{\alpha})$ at $x$ gives the orientation of $P$.
The right-veering property of a boundary-fixing map $\tau: P\to P$ means that for every simple 
arc $\alpha$ with endpoints on $\partial P$, the image $\tau(\alpha)$ is either isotopic to $\alpha$ or lies to the right of $\alpha$ at both endpoints,
once all non-essential intersections between $\alpha$ and  $\tau(\alpha)$
are removed. Now, suppose that the support of the Dehn twist $T_j$ essentially intersects one of the arcs, 
say $\eta_1$. Then the curve  $T_j (\eta_1)$ is not isotopic to $\eta_1$ (see e.g. \cite[Proposition 3.2]{FM}), 
so $T_j (\eta_1)$ lies  to the right of $\eta_1$. Since the composition of right-veering maps is right-veering, 
we can only get curves that lie further to the right of $\eta_1$  after composing with  the other Dehn twists
$T_1, \dots, T_l$.  However, the composition $\phi= T_1 T_2\dots T_j \dots T_l$ fixes $\eta_1$, a contradiction.

\begin{figure}[htb]
		\centering
		\bigskip
		\labellist
		\small\hair 2pt
		\pinlabel $\eta_1$ at 26 20
		\pinlabel $\eta_2$ at 39 20
		\pinlabel $P$ at 8 84
		\pinlabel $P'$ at 133 84
		\pinlabel $D_l$ at 25 36 
		\pinlabel $D_l$ at 148 37
		\endlabellist
		\includegraphics[scale=1.8]{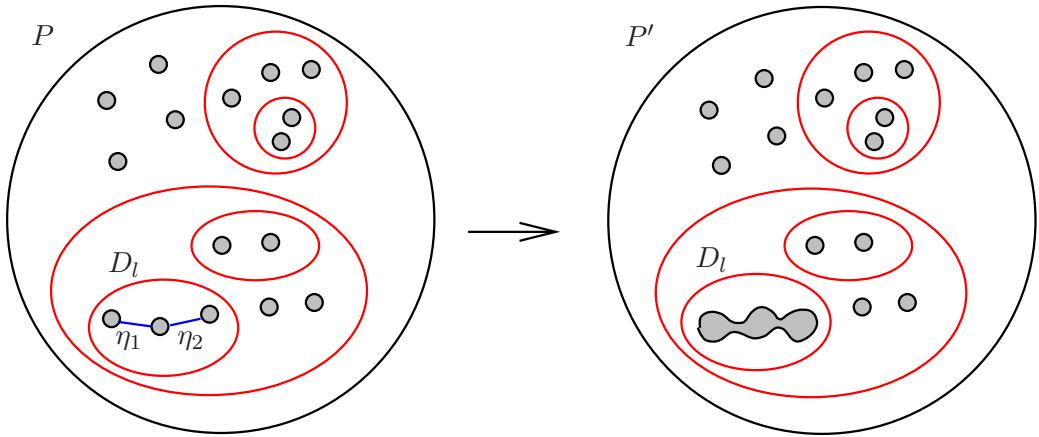}
		\caption{After cutting $P$ along arcs $\eta_1, \eta_2, \dots$, the Dehn twist $D_l$ becomes boundary-parallel in the new surface $P'$.} 
		\label{cutting-disk}
	\end{figure}

Once we know that the support of all the Dehn twists is disjoint from all of the arcs $\eta_1, \dots, \eta_{r-1}$,
we can cut the page $P$ along these arcs, and consider the image of the relation~\eqref{eq:non-bdry} in the resulting cut-up surface $P'$. In $P'$, 
we have that (the induced diffeomorphisms) $T_l$ and $D_l$ are  Dehn twists around the curve homologous to the boundary of the same hole, and therefore, $T_l=D_l$ as Dehn twists in $P'$. It follows that for the Dehn twists (induced by) $D_1, \dots, D_{l-1}$ and $T_1, \dots, T_{l-1}$ in $P'$, we have 
$$
D_1 D_2 \dots D_{l-1}=T_1 T_2 \dots T_{l-1}.
$$
By the induction hypothesis, we can conclude that for each $j=1, \dots, {l-1}$, the Dehn twists $D_j$ and $T_j$ are performed about isotopic curves in $P'$. It follows that each pair $D_j, T_j$  gives the same Dehn twists in $P$, for each $j=1, \dots, l$.
\end{proof}

\begin{proof}[Proof of Theorem~\ref{main}] Under the hypotheses of Theorem~\ref{main}, the contact 3-manifold 
$(Y, \xi)$ is supported by an open book with planar page $P$. Theorems of Wendl and Wendl--Niederkruger then imply that up to blow-up and deformation of the symplectic form, every weak symplectic filling has a Lefschetz fibration whose fiber is given by $P$; the monodromy of the fibration is the monodromy of the open book. The Lefschetz fibration is described by its vanishing cycles, or, equivalently, by a positive  factorization of the monodromy. Lemmas \ref{lem1} and \ref{lem2} show that the positive monodromy factorization is unique.   \end{proof}

\bibliography{references}
\bibliographystyle{alpha}

\end{document}